\documentclass[reqno]{amsart}
\usepackage[foot]{amsaddr}

\usepackage[a4paper,left=3cm,right=3cm,top=3cm,bottom=3cm]{geometry}

\pdfoutput=1

\usepackage{booktabs}
\usepackage{amssymb}
\usepackage[british]{babel}
\usepackage{color}
\usepackage{enumitem}
\usepackage[T1]{fontenc}
\usepackage[utf8]{inputenc}
\usepackage[l2tabu,orthodox]{nag}
\usepackage{tensor}
\usepackage[vcentermath]{youngtab}
\usepackage{tipa}
\usepackage{dsfont}

\usepackage{todonotes}

\usepackage[
		pdftex,
		bookmarks,
		bookmarksopen,
		bookmarksopenlevel=1,
		bookmarksnumbered,
		breaklinks,
		colorlinks,
		linkcolor=linkcolor,
		citecolor=citecolor,
		urlcolor=urlcolor,
	]{hyperref}
	
\hypersetup{
	pdftitle    = {Torsion-free connections of superintegrable systems},
	pdfauthor   = {Andreas Vollmer},
	pdfsubject  = {article},
	pdfcreator  = {\LaTeX{} with `amsart' class},
	pdfproducer = {pdf\LaTeX},
	pdfkeywords = {superintegrable systems}
}

\definecolor{linkcolor}{rgb}{0.5,0.0,0.0}
\definecolor{citecolor}{rgb}{0.0,0.5,0.0}
\definecolor{urlcolor} {rgb}{0.0,0.0,0.5}

\title[Torsion-free connections of superintegrable systems]
	{Torsion-free connections\\ of second-order maximally superintegrable systems}
\subjclass[2010]{
	Primary
	53A20;  %   Projective differential geometry
	Secondary
	53B50,	%	Applications of local differential geometry to the sciences
	70H33.  %  Mechanics of particles and systems - Symmetries and conservation laws, ...
}

\author{Andreas Vollmer$^\sharp$}

\address[$\sharp$]{
	Universität Hamburg\\
	Fachbereich Mathematik\\
	Bundesstr. 55\\
	Hamburg, Germany
}

\numberwithin{equation}{section}

\usepackage{amsthm,thmtools}
\newtheorem{theorem}{Theorem}[section]
\newtheorem{proposition}[theorem]{Proposition}
\newtheorem{lemma}[theorem]{Lemma}

\theoremstyle{definition}
\declaretheorem[name=Definition,qed={\lower-0.3ex\hbox{$\blacktriangleleft$}}]{definition}
\declaretheorem[name=Remark,qed={\lower-0.3ex\hbox{$\blacktriangleleft$}}]{remark}
\declaretheorem[name=Example,qed={\lower-0.3ex\hbox{$\blacktriangleleft$}}]{example}
\newcommand{\RR}{\ensuremath{\mathbb{R}}}

\newcommand{\Ric}{\mathrm{Ric}}
\newcommand{\Riem}{\mathrm{Riem}}

\newcommand{\Proj}{\mathrm{Proj}}
\newcommand{\Scal}{\mathrm{Scal}}

\newcommand{\tr}{\mathrm{tr}}
\renewcommand{\div}{\mathrm{div}}

\setlist[enumerate,1]{label=(\roman*)}

\begin{document}

\begin{abstract}
	Second-order (maximally) conformally superintegrable systems play an important role as models of mechanical systems, including systems such as the Kepler-Coulomb system and the isotropic harmonic oscillator.
	The present paper is dedicated to understanding non- and semi-degenerate systems. 
	We obtain ``projective flatness'' results for two torsion-free connections naturally associated to such systems.
	This viewpoint sheds some light onto the interrelationship of properly and conformally (second-order maximally) superintegrable systems from a geometrical perspective.
	It is shown that the semi-degenerate secondary structure tensor can be viewed as the Ricci curvature of a natural torsion-free connection defined by the primary structure tensor (and similarly in the non-degenerate case).
	It is also shown that properly semi-degenerate systems are characterised, similar to the non-degenerate case, by the vanishing of the secondary structure tensor.
	%
	%These results corroborate that \emph{projective} differential geometry underpins second-order superintegrability alongside conformal geometry, at least in the semi-degenerate case.
\end{abstract}

\maketitle
%\tableofcontents

%============================================================================%
\section{Introduction}\label{sec:intro}

In this paper, second-order (maximally) superintegrable systems are investigated. These play an important role in mathematical physics. For instance, the isotropic harmonic oscillator appears in many models in physics, such as classical mechanics or material sciences. The Kepler-Coulomb system, on the other hand, is a classical model in celestial mechanics, and its quantum version is pivotal in chemistry as a model for the atomic orbits of the hydrogen atom. \emph{Conformally} superintegrable systems are a generalisation of (properly) superintegrable systems, and naturally arise as one studies conformal rescalings of superintegrable systems. Conformal rescalings are, loosely speaking, modifications of the so-called Stäckel transformations (also known as coupling constant metamorphosis).

Here, we focus on so-called \emph{irreducible} systems, which can be characterised geometrically via so-called \emph{structure tensors}.
More specifically we focus on \emph{non-degenerate} and \emph{semi-degenerate} systems, which will be thoroughly introduced below.
A geometric framework, suitable in arbitrary dimension, has recently been developed for irreducible systems \cite{KSV2023,KSV2024}. It is in particular well adapted to studying the algebraic and conformal geometry underpinning \emph{non-degenerate} systems.
Non-degenerate systems are the ones studied most thoroughly in the literature to date, see e.g.~\cite{KKM18}. In particular, they are classified in dimension two and three, e.g. \cite{Kalnins&Kress&Pogosyan&Miller,Kress&Schoebel,KKM05a,KKM05b,KKM05c,KKM06a,Capel}.
Semi-degenerate systems do not exist in dimension two, but examples are known in dimension three (we do not allow mere restrictions of non-degenerate systems, see Remark~\ref{rmk:extendable} below). A full classification is to date still an open problem.

This paper is organised as follows: In Section~\ref{sec:prelim}, we introduce the main terminology needed for the purposes of the paper, including irreducible superintegrable systems and structure tensors for example. An overview of the main results is given in Section~\ref{sec:res}.

In Section~\ref{sec:connection&curvature} we study conformal superintegrability and use the primary structure tensor of a (non- or semi-degenerate) irreducible system in order to introduce a natural torsion-free connection $\nabla$. We obtain an intimate interconnection between the secondary structure tensor and the Ricci curvature of this connection. For semi-degenerate conformally superintegrable systems, in particular, we prove that the induced connection is projectively flat.

In Section~\ref{sec:proper} we then focus on proper superintegrability. We obtain curvature obstructions for the induced connections. In particular, for semi-degenerate systems, we find a Ricci-flatness condition. 
The findings are illustrated with some examples in Section~\ref{sec:ex}. The paper is concluded with a final discussion in Section~\ref{sec:discussion}.

%----------------------------------------------------------------------------%
\subsection{Preliminaries}\label{sec:prelim}

Let $(M,g)$ be a (simply-connected) Riemannian manifold of dimension $n$. We denote its Levi-Civita connection by $\nabla^g$ and its Laplace operator by $\Delta^g$.
The cotangent bundle $T^*M$ is naturally endowed with a symplectic structure and its induced canonical Poisson bracket, which we denote by $\{-,-\}$. Canonical Darboux coordinates will be denoted by $(x,p)$ where $x=(x_1,\dots,x_n)$ are the position and $p=(p_1,\dots,p_n)$ the momenta coordinates.
We consider a \emph{Hamiltonian}, i.e.~a function $H:T^*M\to\RR$, $H(x,p)=g^{-1}_x(p,p)+V(x)$, where $V:M\to\RR$ is a function on $M$, called \emph{potential}. The space of (tracefree) conformal Killing tensors admitted by $g$ is going to be denoted by $\mathcal C(g)$, while the space of (proper) Killing tensors is denoted by $\mathcal K(g)$. We usually drop the specification of $g$ if the underlying metric is clear.

We say that $H$ is (second-order maximally) \emph{properly superintegrable} if there exists a subset $\mathfrak K\subset\mathcal K$ of dimension $\dim(\mathfrak K)\geq 2n-2$ such that there are $2n-1$ functionally independent functions $F^{(r)}$, $r\in\{0,1,\dots,2n-2\}$, with the following properties (the Einstein convention applies):
\begin{enumerate}
	\item $F^{(0)}=H$
	\item $F^{(r)}=K_x(g^{-1}p,g^{-1}p)+W(x)$ where $K$ is a (proper) Killing tensor, i.e.~satisfies
	\[
		\nabla_{(i}K_{jk)}=0\,,
	\]
	where round brackets denote symmetrisation over enclosed indices,
	and where
	\begin{equation}\label{eq:dW.proper}
		\nabla_kW = K\indices{_k^a}\nabla_aV\,.
	\end{equation}
	Note that we interpret $K$ here as a $(1,1)$-tensor by virtue of the metric $g$, using the same symbol by a slight abuse of notation.
\end{enumerate}
We observe that~\eqref{eq:dW.proper} is involutive, i.e.~can be integrated for $W$, if
\begin{equation}\label{eq:ddW.proper}
	d(K(dV)) = 0\,,\quad\text{i.e.~~}
	K\indices{^a_{k,j}}\nabla_aV-K\indices{^a_{j,k}}\nabla_aV
	+K\indices{^a_k}\nabla^2_{aj}V-K\indices{^a_j}\nabla^2_{ak}V=0\,.
\end{equation}
This is the so-called (proper or classical) \emph{Bertrand-Darboux condition} \cite{bertrand_1857,darboux_1901}.
\smallskip

Similarly, we say that $H$ is (second-order maximally) \emph{conformally superintegrable} if there exists a subset $\mathfrak C\subset\mathcal C$ of dimension $\dim(\mathfrak C)\geq 2n-2$ such that there are $2n-1$ functionally independent functions $F^{(r)}$, $r\in\{0,1,\dots,2n-2\}$, with the following properties:
\begin{enumerate}
	\item $F^{(0)}=H$
	\item $F^{(r)}=C_x(g^{-1}p,g^{-1}p)+W(x)$ where $C\in\mathfrak C$ and
	\begin{equation}\label{eq:dW}
		\nabla_kW = C\indices{_k^a}\nabla_aV-\rho_k\,V\,,
	\end{equation}
	with $\rho=\frac2{n+2}\div(C)$.
\end{enumerate}
Note that~\eqref{eq:dW} is involutive, i.e.~can be integrated, if
\begin{equation}\label{eq:ddW}
	d(K(dV))-V\,d\rho-dV\wedge\rho = 0\,.
\end{equation}
This is the so-called \emph{conformal Bertrand-Darboux condition}.

A (conformally) superintegrable system is called \emph{irreducible} if $\mathfrak K$ (respectively $\mathfrak C$) is an irreducible set, i.e.~if the (conformal) Killing tensors do not have a common eigenspace.
For irreducible systems it has been shown that~\eqref{eq:ddW.proper} respectively \eqref{eq:ddW} can be solved for the tracefree Hessian of $V$,
\begin{equation}\label{eq:Wilczynski}
	\nabla^{g,2}_{ij}V
	= T\indices{_{ij}^k}\nabla_kV + V\,\tau_{ij} + \frac1n \Delta^g V\ g_{ij}.
\end{equation}
where $\nabla^{g,2}V$ denotes the covariant Hessian of $V$.
The coefficients $T\indices{_{ij}^k}$ and $\tau_{ij}$ can be shown to be components of tensor fields, and will therefore be called (primary and secondary) structure tensors.
Equation~\eqref{eq:Wilczynski} will be referred to as \emph{Wilczynski equation} of the superintegrable system \cite{KSV2023,KSV2024}.\bigskip
%
%In the present paper, we focus on the following two special cases.\bigskip

By a slight abuse of notation (and following the usual terminology), a (maximal) space of potentials $\mathcal V$ of an irreducible (second-order maximally) conformally superintegrable system is called an \emph{(irreducible superintegrable) potential}. We shall even simply write $V$ for such a potential when we have an explicit parametrisation of this space in mind. However, note that the explicit choice of a parametrisation is not important and just a convenient choice when carrying out explicit computations.

\begin{definition}
	An (irreducible superintegrable) potential is called
	\begin{enumerate}[label=(\roman*)]
		\item \emph{non-degenerate} if~\eqref{eq:Wilczynski} admits an $n+2$ dimensional space of solutions $V$.
		\item \emph{semi-degenerate} if~\eqref{eq:Wilczynski} admits an $n+1$ dimensional space of solutions $V$.
		\item \emph{of higher degeneracy} otherwise.
	\end{enumerate}
\end{definition}

%\noindent We assume that the components of the metric are analytic functions, and we assume $M$ to be an analytic manifold.
In the non-degenerate case, for a point $x_0\in M$, we may consider $V_{x_0}$, $\Delta^g V_{x_0}$ and the components of $dV_{x_0}$ to be linearly independent.
In the semi-degenerate case, on the other hand, there has to be a relation of the form
\begin{equation}\label{eq:semi-degeneracy}
	%	\label{eq:semi-degeneracy.condition}
	\Delta^gV = g^{-1}(s,dV) + \alpha\,V\,,
\end{equation}
where $s$ is a 1-form and where $\Delta^g$ is the Laplace operator of~$g$.
We may then assume $V$ and the components of $\nabla V$ to be linearly independent.
For potentials of higher degeneracy a similar statement holds, but in the current paper, we restrict our attention to non- and semi-degenerate potentials.
To ensure a concise notation, we agree on the following terminology regarding semi-degenerate potentials: we unite~\eqref{eq:Wilczynski} and~\eqref{eq:semi-degeneracy} in the equation
\begin{equation}\label{eq:semi-Wilczynski}
	(\nabla^{g})^2_{ij}V = D\indices{_{ij}^k}\nabla_kV + \eta_{ij}\,V\,.
\end{equation}
where we introduce the \emph{(primary) semi-degenerate structure tensor}
$$ D\indices{_{ij}^k}:= T\indices{_{ij}^k}+\frac1n\,g_{ij}\otimes s^k$$
and the \emph{secondary semi-degenerate structure tensor}
$$ \eta_{ij}=\tau_{ij}+\frac1n\,\alpha\,V\,g_{ij}. $$
Similarly, for a non-degenerate potential, we call $T$ and $\tau$ the (primary respectively secondary) \emph{non-degenerate structure tensor}. Note that these tensors are uniquely determined by the space $\mathfrak C$ (or $\mathfrak K$, respectively).

\begin{remark}\label{rmk:extendable}
	Note that we consider \emph{non-extendable} potentials in the cases of semi-degeneracy and higher degeneracy as we require $\mathcal V$ to be a maximal space of potentials. A weaker definition would allow a semi-degenerate potential to be merely the subset of a non-degenerate potential, but here we intend to exclude this phenomenon in order to ensure a focused exposition.
\end{remark}

\subsection*{Non-degeneracy}
Non-degenerate systems are the ones that have, to date, been studied most thoroughly. They are classified in dimensions two and three, including a classification of their conformal classes (i.e.\ St\"ackel classes) \cite{Kalnins&Kress&Pogosyan&Miller,Kress07,Kress&Schoebel,Kalnins&Kress&Miller,Capel,Capel_phdthesis,KKM05a,KKM05b,KKM05c,KKM06a}.
Non-degeneracy is the case of maximal dimension for the space of compatible potentials.
It is shown in~\cite{KSV2023,KSV2024} that the Ricci identity for~\eqref{eq:Wilczynski} implies the system of partial differential equations
\begin{subequations}\label{eq:non-deg.prolongation}
	\begin{align}
		(\nabla^g)^2_{ij} V &= T\indices{^k_{ij}}V_{,k} + \tau_{ij}\,V + \frac1n\,\Delta V\, g_{ij}
		\\
		\frac{n-1}{n}\nabla_k\Delta V &= q_{ka}V^{,a} + \div(\tau)_k\,V
%		\\
%		\intertext{where}
%		q_{ij} &= T\indices{^a_{ij,a}}+T\indices{_{ai}^b}T_{baj}-\Ric^g_{ij}+\tau_{ij}\,.
	\end{align}
	where
	\[
		q_{ij} = T\indices{^a_{ij,a}}+T\indices{_{ai}^b}T_{baj}-\Ric^g_{ij}+\tau_{ij}\,.
	\]
\end{subequations}
This implies that, locally, the potential $V$ can be recovered from the knowledge of the values of $V(x_0)$, $dV_{x_0}$ and $(\Delta V)_{x_0}$ in a point $x_0\in M$. The system~\eqref{eq:non-deg.prolongation}, expressing all higher derivatives in terms of lower ones is called a \emph{closed prolongation}.

A particular subclass of non-degenerate systems are so-called \emph{abundant} systems. These admit a space of $\frac{n(n+1)}{2}-1$ tracefree conformal Killing tensors compatible with the $(n+2)$-dimensional space of potentials.
In dimension (two and) three all non-degenerate systems are abundant (in dimension two this is a trivial statement). Abundant systems in arbitrarily high dimension have been investigated in \cite{KSV2023,KSV2024}. In particular, in \cite{KSV2024} obstructions have been found for abundant systems in dimensions $n\geq4$, which are not present in lower dimensions.

For the purposes here, the conditions of superintegrability for non-degenerate systems obtained in \cite{KSV2023,KSV2024} are particularly important. Although the mentioned references focus on abundant systems, they obtain (among other results) the superintegrability conditions for non-degenerate systems.

\subsection*{Semi-degeneracy}
The semi-degenerate case is the case of sub-maximal dimension of $\mathcal{V}$. Note that~\eqref{eq:semi-Wilczynski} already forms a closed prolongation whose initial values are the $n+1$ constants $V(x_0)$ and $dV_{x_0}$.

%----------------------------------------------------------------------------%
\subsection{Main results}\label{sec:res}

Having reviewed the theory of irreducible second-order systems as introduced in \cite{KSV2023,KSV2024}, we now use the primary structure tensors $T$ and $D$, respectively, to define connections on $TM$.
In the non-degenerate case, we consider~\eqref{eq:Wilczynski} and define the connection $\nabla^T$ by
\[
\nabla^T_XY = \nabla^{g}_XY - T(X,Y)\,.
\]
In the semi-degenerate case, on the other hand, we introduce the connection $\nabla^D$ by
\[
\nabla^D_XY = \nabla^{g}_XY - D(X,Y)\,.
\]
To ensure a concise notation, we will typically simply write $\nabla$ instead of $\nabla^T$ or $\nabla^D$, since there is no risk of confusion as these connections appear in clearly separated contexts.
We observe that $\nabla$ is torsion-free, since $\nabla^g$ is torsion-free and $T_{ijk}=g_{ka}T\indices{_{ij}^a}$ as well as $D_{ijk}=g_{ka}D\indices{_{ij}^a}$ are symmetric in the first two indices.
Furthermore, it is easy to verify:
\begin{itemize}
	\item For a non-degenerate structure tensor $T$, the Laplace operators $\Delta$ and $\Delta^g$ coincide since $T_{ijk}$ is tracefree in the first two indices.
	\item For a semi-degenerate structure tensor $D$, the Laplace operators satisfy
	\[
	\Delta^gV = \Delta V + dV(s)\,,
	\]
	where $s$ is the 1-form with components $s_k=g^{ab}T_{abk}$.
\end{itemize}
The curvature tensor of $\nabla$ will be denoted by $R=R^\nabla$, and we introduce its Riemann tensor $\Riem^{\nabla,g}$ by lowering the upper index of $R$ by virtue of the metric $g$. Moreover, we denote the Ricci tensor of $\nabla$ by $\Ric^\nabla_{ij}=R\indices{^b_{ibj}}$ and introduce the scalar curvature $\Scal^{\nabla,g}=(R^\nabla)\indices{^{ab}_{ab}}$. The superscript will usually be dropped to ensure a concise notation.

Our first set of main results concerns non-degenerate systems.
%Note that the torsion-freeness of $\nabla^T$ follows immediately from the symmetries of $T$.
\begin{theorem}\label{thm:first}
	Let $(M^n,g)$ with $n\geq3$ and let $V$ be a non-degenerate potential with primary structure tensor $T$ and secondary structure tensor $\tau$.
	We introduce the torsion-free connection
	\[ \nabla^T := \nabla^g - T \]
	where $\nabla^g$ is the Levi-Civita connection.
	Then:
	\begin{enumerate}[label=(\roman*)]
		\item
		The secondary structure tensor $\tau$ is obtained from the curvature $R$ of~$\nabla^T$,
		\[
			\tau_{ij} = \frac1{n(n-2)}\left( g_{ib}g^{ac}R\indices{^{b}_{caj}}-g_{ij}g^{ac}R\indices{^{b}_{cab}} -(n-1)R\indices{^b_{ibj}}\right)\,.
		\]
		\item
		The curvature $R$ of $\nabla^T$ satisfies
		\begin{equation*}
			R\indices{^b_{ijk}}
			= \tau_{ij}\delta_k^b-\tau_{ik}\delta_j^b
			+\frac1{n-1}\,g_{ij}\left( g^{ab}\tau_{ak}+B\indices{^b_k}\right)
			-\frac1{n-1}\,g_{ik}\left( g^{ab}\tau_{aj}+B\indices{^b_j}\right)\,,
		\end{equation*}
		where $B\indices{^a_k}=g^{ij}R\indices{^a_{ijk}}$.
		In particular, if the system is proper ($\tau=0$), then the connection $\nabla^T$ satisfies
		$$ \Riem\indices{_{aijk}} %= g_{ab}R\indices{^b_{ijk}}
		= g_{ij}\Ric\indices{_{ak}}-g_{ik}\Ric_{aj} $$
		where $\Ric$ is the Ricci tensor of $\nabla^T$ and $\Riem_{aijk}=g_{ab}R\indices{^b_{ikl}}$.
	\end{enumerate}
\end{theorem}

The second main outcome is an analogous result for semi-degenerate systems.
\begin{theorem}\label{thm:second}
	Let $(M^n,g)$ with $n\geq3$ and let $V$ be a semi-degenerate potential with primary structure tensor $D$ and secondary structure tensor $\eta$.
	We introduce the torsion-free connection
	\[ \nabla^D := \nabla^g - D \]
	where $\nabla^g$ is the Levi-Civita connection.
	Then:
	\begin{enumerate}[label=(\roman*)]
		\item
		The secondary structure tensor $\eta$ is naturally obtained from the Ricci tensor $\Ric$ of~$\nabla^D$,
		\[
		\tau_{ij} = \frac1{n-1}\Ric_{ij}\,.
		\]
		\item
		The connection $\nabla^D$ is projectively flat.
		\item 
		If the system is proper ($\eta=0$), then $\nabla^D$ is flat and the potentials are eigenfunctions of the Laplace operator of $\nabla^D$.
	\end{enumerate}
\end{theorem}

We will also obtain:
\begin{proposition}\label{prop:proper}
	The secondary structure tensor of a non-degenerate conformally superintegrable system is proper ($\tau=0$) if and only if the system corresponds to a (properly) non-degenerate superintegrable system.
	
	The secondary structure tensor of a semi-degenerate conformally superintegrable system is proper ($\eta=0$) if and only if the system corresponds to a (properly) semi-degenerate superintegrable system.
\end{proposition}
The first part of this claim has already been stated and proven in~\cite{KSV2023}.

%============================================================================%
\section{The induced connection}\label{sec:connection&curvature}

%----------------------------------------------------------------------------%
\subsection{The non-degenerate case}
We abbreviate $\nabla=\nabla^T$ in this paragraph.
In the non-degenerate case we have a Wilczynski equation of the form~\eqref{eq:Wilczynski}. Bringing the first term on the right hand side to the left side of the equation, we obtain
\begin{equation}\label{eq:non-deg.Wilczynski}
	\nabla^2_{ij}V = \tau_{ij}\,V+\frac1n\,\Delta V\,g_{ij}
\end{equation}
in terms of the induced connection $\nabla$.
Using the (contracted) Ricci identity for~\eqref{eq:non-deg.Wilczynski}, or rewriting~\eqref{eq:non-deg.prolongation} in terms of $\nabla$ and its curvature, we obtain
\begin{equation}\label{eq:non-deg.Wilczynski.prolonged}
	\frac{n-1}{n}\nabla_k\Delta V = \left(\tau_{ka}-\Ric_{ka}\right)\nabla^aV + \div(\tau)_k\,V+\frac1n\div(g)_k\,\Delta V\,.
\end{equation}
Subsequently, we may resubstitute~\eqref{eq:non-deg.Wilczynski.prolonged} into the Ricci identity for~\eqref{eq:non-deg.Wilczynski}. Moreover, the Ricci identity for~\eqref{eq:non-deg.Wilczynski.prolonged} (i.e., the symmetry of the Hessian) must also hold.

We recall the following fact stated in Proposition~\ref{prop:proper}. 
It is proven in Lemma 3.15 (``$\Rightarrow$'') and Corollary 5.3 (``$\Leftarrow$'') of \cite{KSV2024}.
\begin{lemma}
	A non-degenerate system corresponds to a (properly) non-degenerate (second-order maximally) superintegrable system if and only if $\tau=0$.
\end{lemma}
Note that this lemma is non-trivial since in a (properly) non-degenerate (second-order maximally) superintegrable system the potential is compatible with (actual) Killing tensors instead of (trace-free) conformal Killing tensors.\medskip

We shall now prove the first main result, Theorem~\ref{thm:first}, by proving each item in the claim individually.
%We begin with some general considerations.
%
Using the Ricci identity for~\eqref{eq:non-deg.Wilczynski}, we have
\begin{align*}
	R\indices{^b_{ijk}}\nabla_bV
	&= (\nabla_k\tau_{ij}-\nabla_j\tau_{ik})V
	+(\tau_{ij}g_{ka}-\tau_{ik}g_{ja})\nabla^a V
	\\
	&\quad +\frac1n(\nabla_kg_{ij}-\nabla_jg_{ik})\Delta V
	+\frac1n(g_{ij}\nabla_k\Delta V-g_{ik}\nabla_j\Delta V)\,,
\end{align*}
and contracting in $(i,j)$, we obtain
\[
\frac{n-1}{n}\nabla_k\Delta V = (\tau_{ak}+B_{ak})\nabla^aV + \mathtt{t}_{k} V - \frac1n T\indices{_{ka}^a}\Delta V\,.
\]
where we introduce
\begin{equation}
	B\indices{^b_k}:=g^{ij}R\indices{^b_{ijk}}.
\end{equation}
and
\begin{equation}
	\mathtt{t}_k:=g^{ij}(\nabla_k\tau_{ij}-\nabla_j\tau_{ik})\,.
\end{equation}
We have also used the identity
\begin{align*}
	\nabla_kg_{ij}-\nabla_jg_{ik} &= \frac{1}{n-1}(T\indices{_{ja}^a}g_{ik}-T\indices{_{ka}^a}g_{ij})\,.
\end{align*}
Resubstituting $\nabla_k\Delta V$, we find
\begin{align*}
	\Riem_{aijk}\nabla^aV
	&= (\nabla_k\tau_{ij}-\nabla_j\tau_{ik})V
	+(\tau_{ij}g_{ka}-\tau_{ik}g_{ja})\nabla^a V
	+\frac1n(\nabla_kg_{ij}-\nabla_jg_{ik})\Delta V
	\\
	&\quad +\frac1{n-1}\Big[
	g_{ij} (\tau_{ak}+B_{ak})\nabla^aV + g_{ij}\mathtt{t}_k V - \frac1ng_{ij}T\indices{_{ka}^a}\Delta V
	\\ &\qquad\qquad\qquad
	-g_{ik} (\tau_{aj}+B_{aj})\nabla^aV + g_{ik}\mathtt{t}_{j} V - \frac1ng_{ik}T\indices{_{ja}^a}\Delta V
	\Big]\,,
\end{align*}
and, rewriting according to the derivatives of $V$,
\begin{multline}\label{eq:nondeg.ricciIdentity}
	\Big[
	\tau_{ij}g_{ka}-\tau_{ik}g_{ja}+\frac1{n-1}g_{ij}\tau_{ak}-\frac1{n-1}g_{ik}\tau_{aj}
	\\
	-\Riem_{aijk}+\frac1{n-1}g_{ij}B_{ak}-\frac1{n-1}g_{ik}B_{aj}
	\Big]\nabla^aV
	\\
	+\Big[
	\nabla_k\tau_{ij}-\nabla_j\tau_{ik}+\frac1{n-1}g_{ij}\mathtt{t}_{k}-\frac1{n-1}g_{ik}\mathtt{t}_j
	\Big]V
	\\
	+\frac1n\Big[
	\underbrace{ \nabla_kg_{ij}-\nabla_jg_{ik}+\frac1{n-1}g_{ij} T\indices{_{ka}^a}-\frac1{n-1}g_{ik}T\indices{_{ja}^a} }_{=0}
	\Big]\Delta V = 0
\end{multline}
As indicated, the coefficient of $\Delta V$ is zero, and we are therefore left with the coefficients of $\nabla V$ and $V$. Note that these have to vanish independently due to non-degeneracy \cite{KKM05a,KKM05c,KSV2023,KSV2024}.
Consider first the coefficient of $\nabla V$, i.e.~the condition
\begin{multline}\label{eq:dv.coeff.nondeg}
	\tau_{ij}g_{ka}-\tau_{ik}g_{ja}+\frac1{n-1}g_{ij}\tau_{ak}-\frac1{n-1}g_{ik}\tau_{aj}
	\\
	-\Riem_{aijk}+\frac1{n-1}g_{ij}B_{ak}-\frac1{n-1}g_{ik}B_{aj}
	= 0\,.
\end{multline}

We now prove the first claim of Theorem~\ref{thm:first}.

\begin{proposition}\label{prop:tau.non-deg}
	Let $(M,g,T,\tau)$ define a non-degenerate irreducible system of dimension~$n\geq3$ satisfying~\eqref{eq:non-deg.Wilczynski}.
	Then the secondary structure tensor satisfies
	\begin{equation}\label{eq:tau.nondeg}
		\tau_{ij}=\frac{1}{n(n-2)}\left(
			B_{ij}-g_{ij}\tr(B)-(n-1)\Ric_{ij}
		\right)\,,
	\end{equation}
	where $\Ric$ is the Ricci curvature of $\nabla$, $\Ric_{ij}=R\indices{^b_{ibj}}$,
	and where $B_{ij}=g_{ia}B\indices{^a_j}$ with
	\[
		B\indices{^b_k}:=g^{ij}R\indices{^b_{ijk}}\,.
	\]
\end{proposition}
\noindent 
Note that $\Scal^{\nabla,g}=-\tr(B)$, and thus $\tau_{ij}$ is trace-free, as expected.
\begin{proof}
	Contracting~\eqref{eq:dv.coeff.nondeg} in $(k,a)$, we obtain, with $\tr(B):=g^{ij}B_{ij}$,
	\begin{equation}\label{eq:pre-formula.tau.nondeg}
		n(n-2)\tau_{ij} = B_{ij}-g_{ij}\tr(B)-(n-1)\Ric_{ij}\,.
	\end{equation}
	Due to $n\geq3$, we can solve this condition for $\tau$, yielding the claim.
%	Contracting~\eqref{eq:dv.coeff.nondeg} in $(i,k)$ instead, using $g$, we obtain
%	\begin{equation*}
%		0 = (B_{ij}-g_{ij}\tr(B))-(n-1)\Ric_{ij}\,.
%	\end{equation*}
\end{proof}

With this at hand, we immediately prove the second part of Theorem~\ref{thm:first}.

\begin{proposition}\label{prop:curv.non-deg}
	Let $(M,g,T,\tau)$ define a non-degenerate irreducible system of dimension~$n\geq3$ satisfying~\eqref{eq:non-deg.Wilczynski}.
	Then the curvature tensor $\Riem_{aijk}=g_{ab}R\indices{^b_{ijk}}$ of $\nabla$ satisfies, where $\Ric$ is the Ricci tensor of $\nabla$ and $\Scal=g^{ij}\Ric_{ij}$,
	\begin{multline}\label{eq:Weyl.flatness.nondeg}
		\Riem_{aijk}
		= \frac1{(n-1)(n-2)}\left( g_{ij}g_{ka}-g_{ik}g_{ja} \right)\,\Scal
		\\
		+\frac1{n(n-2)}\bigg( g_{ij}((n-1)B_{ak}-\Ric_{ak}) - g_{ik}((n-1)B_{aj}-\Ric_{aj}) \\
		+ (B_{ij}-(n-1)\Ric_{ij})g_{ak} - (B_{ik}-(n-1)\Ric_{ik})g_{aj} \bigg)
		\,,
	\end{multline}
	with $B_{ij}=g_{ia}B\indices{^a_j}$ as before, and
	\begin{equation}\label{eq:Cotton.flatness.nondeg}
		\nabla_k\tau_{ij}-\nabla_j\tau_{ik}+\frac1{n-1}g_{ij}\mathtt{t}_{k}-\frac1{n-1}g_{ik}\mathtt{t}_j=0\,,
	\end{equation}
	where $\mathtt{t}_k=g^{ij}(\nabla_k\tau_{ij}-\nabla_j\tau_{ik})$.
\end{proposition}
\noindent Note the formal analogy with the equations for projective and conformal flatness. The left hand side in~\eqref{eq:Cotton.flatness.nondeg} resembles the Cotton tensor, while the one in \eqref{eq:Weyl.flatness.nondeg} is reminiscent of the Weyl tensor.
\begin{proof}
	We have found in~\eqref{eq:dv.coeff.nondeg} that
%	\begin{multline*}
%		0 = -\Riem_{aijk}-\frac1{n-2}g_{ka}\mathring{\Ric}_{ij}+\frac1{n-2}g_{ja}\mathring{\Ric}_{ik}
%		-\frac1{n-2}g_{ij}\mathring{\Ric}_{ak}+\frac1{n-2}g_{ik}\mathring{\Ric}_{aj}
%		\\
%		-\frac{R}{n(n-1)}(g_{ij}g_{ka}-g_{ik}g_{aj})
%		= \Weyl_{aijk}\,.
%	\end{multline*}
	\begin{multline*}
		\Riem_{aijk}=\tau_{ij}g_{ka}-\tau_{ik}g_{ja}+\frac1{n-1}g_{ij}\tau_{ak}-\frac1{n-1}g_{ik}\tau_{aj}
		+\frac1{n-1}g_{ij}B_{ak}-\frac1{n-1}g_{ik}B_{aj}\,.
	\end{multline*}
	Resubstituting~\eqref{eq:pre-formula.tau.nondeg} and collecting terms, we arrive at the first claim.
	The second part of the claim similarly is obtained from the coefficient of $V$ in~\eqref{eq:nondeg.ricciIdentity}.
%
%	We now consider the coefficient of $V$ in~\eqref{eq:nondeg.ricciIdentity},
%	\begin{equation}\label{eq:V.coeff.nondeg}
%		\nabla_k\tau_{ij}-\nabla_j\tau_{ik}+\frac1{n-1}g_{ij}\nabla^a\tau_{ak}-\frac1{n-1}g_{ik}\nabla^a\tau_{aj} = 0\,,
%	\end{equation}	
%	and obtain
%	\[
%		(\nabla_k\mathring\Ric_{ij}-\nabla_j\mathring\Ric_{ik})_\circ = 0
%	\]
%	where $\circ$ denotes trace-freeness. It is straightforwardly checked that it is equivalent to $C=0$ for the Cotton tensor of $\nabla$.
%	This confirms the claim.
\end{proof}

\begin{proposition}\label{prop:proj.non-deg}
	Let $(M,g,T,\tau)$ define a non-degenerate irreducible system of dimension~$n\geq3$ satisfying~\eqref{eq:non-deg.Wilczynski}.
	Assume $\tau=0$.
	Then $\nabla$ satisfies
	\begin{equation}\label{eq:proj.flatness.nondeg}
		\Riem_{aijk} =  \frac1{n-1}g_{ij}B_{ak}-\frac1{n-1}g_{ik}B_{aj}\,.
	\end{equation}
	with $B_{ij}=(n-1)\Ric_{ij}+g_{ij}\Scal$ where $\Ric=\Ric^{\nabla}$ and $\Scal=\Scal^{\nabla,g}$ for $\nabla=\nabla^T$.
\end{proposition}
\begin{proof}
	In the case $\tau_{ij}=0$, we obtain from Proposition~\ref{prop:curv.non-deg} that the remaining terms on the left hand side of \eqref{eq:dv.coeff.nondeg} vanish. Recalling~\eqref{eq:tau.nondeg}, the claim is obtained.
\end{proof}

\noindent Note the formal analogy  of~\eqref{eq:proj.flatness.nondeg} with the formula of the projective (Weyl) curvature tensor.

%----------------------------------------------------------------------------%
\subsection{The semi-degenerate case}

In this paragraph, we abbreviate $\nabla=\nabla^D$. We proceed with the semi-degenerate case. Rewriting~\eqref{eq:semi-Wilczynski} analogously to~\eqref{eq:Wilczynski}, we find
\begin{equation}\label{eq:semi-deg.Wilczynski}
	\nabla^2_{ij}V = \tau_{ij}V\,.
\end{equation}
Note that this is already a prolongation in closed form.
Contracting, we have in particular that
\[
	\Delta V = \alpha V
\]
where $\alpha=\tr(\tau)$.
Similar to the non-degenerate case, the tensor field $\tau$ characterises, among the conformal systems, those systems that are proper, i.e.~a suitable choice of the traces transforms the associated conformal Killing tensors into actual Killing tensors.
\begin{lemma}
	A semi-degenerate conformal system is proper if and only if $\eta=0$.
	
\end{lemma}
\begin{proof}
	The direction ``$\Rightarrow$'' has partially been proven in~\cite[Lemma 3.15]{KSV2024}, which obtains $\tau=0$ under the hypothesis.
	Reviewing this proof, we see that it actually implies $\rho=0$, thus we infer $d\rho=0$, which by comparison with~\eqref{eq:semi-Wilczynski} and~\eqref{eq:ddW} confirms $\eta=0$.
	
	We therefore turn to the direction ``$\Leftarrow$''. We proceed as in~\cite[Cor.~5.3]{KSV2024}, where the corresponding proof was given for the non-degenerate case.
	Let $C_{ij}$ be a trace-free conformal Killing tensor of the conformally superintegrable system. We need to find a function $\lambda$ such that $K_{ij}=C_{ij}+\frac1n\,g_{ij}\lambda$ is a proper Killing tensor, i.e.\ it satisfies the Bertrand-Darboux condition~\eqref{eq:ddW.proper}.
	We proceed in two steps. First we show that $d\omega=0$. Then we prove that this leads to a properly superintegrable system.
	
	For the first step, resubstitute~\eqref{eq:semi-Wilczynski} into~\eqref{eq:ddW}.
	We obtain
	\begin{multline*}
		0 = \nabla_kC\indices{^a_j}\nabla_aV - \nabla_jC\indices{^a_k}\nabla_aV
		\\
		+ C\indices{^a_j}D_{kab}\nabla_bV - C\indices{^a_k}D_{jab}\nabla_bV
		+ C\indices{^a_j}\eta_{ka}V - C\indices{^a_k}\eta_{ja}V
		\\
		+ \nabla_jV\rho_i - \nabla_iV\rho_j
		+ V\nabla_j\rho_i - V\nabla_i\rho_j\,.
	\end{multline*}
	Due to semi-degeneracy, the coefficients of $\nabla V$ and $V$ have to vanish independently.
	Consider first the coefficient of $V$. It yields
	\[
	2\,d\omega_{ij} = \young(i,j)\,\omega_{i,j}
	= \young(i,j)\,C\indices{^a_j}\eta_{ia}
	= \young(i,j)\,C\indices{^a_j}\left(\tau_{ia}+\frac1ng_{ia}\alpha\right)
	= \young(i,j)\,C\indices{^a_j}\tau_{ia}
	= \young(i,j)\,K\indices{^a_j}\tau_{ia}
	\]
	for any trace correction $\lambda$.
	We infer that $\omega$ is exact if $\tau_{ij}=0$, i.e.~$\omega=d\lambda$ for some function $\lambda$.
	Now let $K_{ij}=C_{ij}+\frac1n\,g_{ij}\lambda$ with this specific function~$\lambda$. We conclude, using the trace-freeness of $C_{ij}$,
	\begin{equation}\label{eq:BD.proper}
		\left( d(KdV) \right)_{ij}
		= \frac12\,\young(i,j)\,\left(
		K_{ia,j}V^{,a}
		+ K\indices{_i^a} V_{,aj}
		\right)
		= 0\,,
	\end{equation}
	due to the conformal Bertrand-Darboux condition~\eqref{eq:ddW}.
	So $K_{ij}$ satisfies the proper Bertrand-Darboux equation~\eqref{eq:ddW.proper}.
%	This proves the claim, but one might have noticed that in the above computation we have only used $\tau=0$. However, considering~\eqref{eq:ddW} we observe that $\eta=0$ is equivalent to it (due to the $V\,d\rho$ term in~\eqref{eq:ddW}).
\end{proof}

We now prove Theorem~\ref{thm:second}, again proceeding on a point by point basis.
We begin with some general considerations analogously to the non-degenerate case.
%Note, however, that the dimensional restriction can be relaxed in the semi-degenerate case.
We denote the curvature tensors again without explicit reference to $\nabla$ as there is no risk of confusion.
Due to the Ricci identity for~\eqref{eq:semi-deg.Wilczynski},
\begin{equation}\label{eq:start.semideg}
	(\Riem_{bijk}-\eta_{ij}g_{kb}+\eta_{ik}g_{jb})\nabla^bV = (\nabla_k\eta_{ij}-\nabla_j\eta_{ik})V\,.
\end{equation}
Due to semi-degeneracy, the coefficients of $\nabla V$ and $V$ have to vanish independently.

\begin{proposition}\label{prop:eta.via.Ric}
	Let $(M,g,D,\eta)$ define a semi-degenerate irreducible system satisfying~\eqref{eq:semi-deg.Wilczynski}.
	Then
	\begin{equation}\label{eq:Eta}
		\eta=-\frac{1}{n-1}\Ric,
	\end{equation}
	where $\Ric$ is the Ricci curvature of $\nabla$.
\end{proposition}
\begin{proof}
	From the coefficient of $\nabla V$ in~\eqref{eq:start.semideg}, we obtain, contracting in $(b,j)$,
	\[
		\Ric_{ik}-\eta_{ik}+n\eta_{ik} = 0
	\]
	i.e.
	\[
		\Ric_{ij}=-(n-1)\eta_{ij}\,.
	\]
	This proves the claim.
\end{proof}

\begin{proposition}\label{prop:nabla.semi-deg.proj.flat}
	Let $(M,g,D,\eta)$ define a semi-degenerate irreducible system satisfying~\eqref{eq:semi-deg.Wilczynski}.
	Then $\nabla$ is projectively flat.
\end{proposition}
\begin{proof}
	Resubstituting $\eta$ into~\eqref{eq:start.semideg}, using Proposition~\ref{prop:eta.via.Ric}, we infer
	\begin{equation}\label{eq:Riem.semideg}
		\Riem_{ijkl} - \frac1{n-1}(\Ric_{jl}g_{ik}-\Ric_{jk}g_{il})
		= 0\,.
	\end{equation}
	The left hand side in this equation is precisely the so-called \emph{projective Weyl tensor}, thus implying that $\nabla$ is projectively flat since $n\geq3$.
%	However, this condition is not sufficient in the case of dimension $n=2$. However, in this case, consider the coefficient of $V$ in the above Ricci identity. It follows that $\Ric$ is Codazzi with respect to $\nabla$ (which holds in any dimension $n\geq2$) and this confirms that $\nabla$ is projectively flat in any dimension $n\geq2$.
\end{proof}

\begin{proposition}
	Let $(M,g,D,\eta)$ define a semi-degenerate irreducible system satisfying~\eqref{eq:semi-deg.Wilczynski}.
	Then the potential $V$ satisfies
	\[
		\Delta V = -\frac{1}{n-1}\,\Scal\,V\,.
	\]
	with the Laplace operator $\Delta:=\Delta^{\nabla,g}=\tr(g^{-1}\nabla^2)$ and the scalar curvature $\Scal:=\Scal^{\nabla,g}$.
\end{proposition}
\begin{proof}
	Consider the contraction of~\eqref{eq:semi-deg.Wilczynski}, and substitute $\eta$ using Proposition~\ref{prop:eta.via.Ric}.
	The claim is obtained.
\end{proof}

\section{Proper systems}\label{sec:proper}

We are now going to consider the case of irreducible \emph{properly} superintegrable systems that are either non-degenerate or semi-degenerate.

\begin{theorem}\label{thm:proper.admissible.nondeg}
	For a proper, non-degenerate system the curvature tensor of the induced connection $\nabla=\nabla^T$ satisfies
	\[
				R\indices{^a_{ijk}} =  \frac1{n-1}g_{ij}B\indices{^a_k}-\frac1{n-1}g_{ik}B\indices{^a_j}\,.
	\]
	and is Ricci-symmetric. %Moreover, the trace of $T$ is a closed 1-form.
\end{theorem}
%By the trace of $T$ we here mean the 1-form with components $t_k=T\indices{_{ka}^a}$.
\begin{proof}
	The first claim was already obtained in~\eqref{prop:proj.non-deg}.
	For the second claim, we make use of a result obtained in \cite{KSV2023}, namely that the trace $t_i:=T\indices{_{ia}^a}$ of a non-degenerate structure tensor is closed.
	Now let $\omega$ be the Riemannian volume form on $(M,g)$. We need to show that there is a function $f$ on $M$ such that
	\[
		\nabla_X(f\omega)=0\,.
	\]
	Now we have, for a basis $Y_1,\dots,Y_n$ of $TM$,
	\begin{align*}
		\nabla_X(f\omega)(Y_1,\dots,Y_n) &= \left[ \nabla_X(f)+f\,\tr\,T(X,-) \right]\,\omega(Y_1,\dots,Y_n)
		\\
		&= \left[ \nabla_X(f)+f\,t(X) \right]\,\omega(Y_1,\dots,Y_n)
	\end{align*}
	and conclude that such $f$ exists on a simply connected space and in particular locally.
\end{proof}

\noindent We note that Ricci-symmetry follows analogously in the conformal case using the analogous result of closedness of the trace $T\indices{_{ia}^a}$ from \cite{KSV2024}.
We turn back to the semi-degenerate case.

\begin{theorem}\label{thm:proper.admissible.semideg}
	For a proper, semi-degenerate system the induced connection $\nabla=\nabla^D$ is flat.
	%The 1-form with components $d_k=D\indices{_{ka}^a}$ is closed.
\end{theorem}
\noindent We remark that the Ricci-symmetry of $\nabla$, analogous to the non-degenerate case, follows already from a prior result by Jeremy Nugent \cite{Jeremy} who proves the closedness of $D\indices{_{ka}^a}$.
\begin{proof}
	The flatness immediately follows, analogously to the non-degenerate case, from Proposition~\ref{prop:nabla.semi-deg.proj.flat}, resubstituting $\eta=0$ into~\eqref{eq:Riem.semideg}.
%	To prove that $d_k$ is closed, we write out $R=R^\nabla$ in terms of $\Riem^g$ and $D$. It follows that
%	\[
%		R\indices{^a_{ijk}} = (R^g)\indices{^a_{ijk}} + \nabla_jD\indices{_{ik}^a}-\nabla_kD\indices{_{ij}^a}
%								+ D\indices{_{ik}^b}D\indices{_{bj}^a} - D\indices{_{ij}^b}D\indices{_{bk}^a}
%							=0.
%	\]
%	We therefore infer the Ricci tensor
%	\[
%		\Ric_{ik} = (R^g)_{ik}
%					+ \nabla_aD\indices{_{ik}^a} -\nabla_kd_i
%					+ D\indices{_{ik}^a}d_a - D\indices{_i^{ab}}D_{kba}
%				= 0 .
%	\]
%	Antisymmetrising, we find
%	\[
%		\nabla_id_k-\nabla_kd_i = 0.
%	\]
%	This confirms the proof (noting that $D$ is symmetric in its lower indices).
\end{proof}

%Due to its flatness, the induced connection in the non-degenerate case is Ricci-symmetric and thus $D\indices{_{ka}^a}$ should be closed by a similar argument as in the proof of Theorem~\ref{thm:proper.admissible.nondeg}. This is consistent with a result by Jeremy Nugent \cite{Jeremy} who proves the closedness of $D\indices{_{ka}^a}$ in his PhD thesis.

\noindent To avoid misunderstandings, we note that the flat connection $\nabla$ in the theorem usually does not define a Hessian structure. The theorem states that a properly semi-degenerate system induces a flat structure $\nabla$ on the Riemannian manifold $(M,g)$. A Hessian structure, however, would require the metric to satisfy $g=\nabla df$ for some function $f$ on $M$, which is equivalent to requiring that $g_{ia}D\indices{_{jk}^a}$ is symmetric \cite{Shima}. Note that $D\indices{_{jk}^a}$ is only required to be symmetric in~$(j,k)$.

\begin{proposition}
	Let $(M,g,D,\eta)$ define a semi-degenerate irreducible system satisfying~\eqref{eq:semi-deg.Wilczynski}.
	Assume $\eta=0$.
	Then % the connection $\nabla$ is flat and
	the potentials are eigenfunctions for the eigenvalue $0$.
\end{proposition}
\begin{proof}
	Resubstitute $\eta=0$ into~\eqref{eq:Riem.semideg}.
\end{proof}

%\begin{example}
%	We consider a conformal example from~\cite{Escobar-Ruiz&Miller}, namely the system labeled `I'.
%	The potential for the system `I' is
%	\[
%			V= a_0\,\frac{x^2 + y^2 + z^2 - 1}{(x^2 + y^2 + z^2 + 1)^2\sqrt{x^2 + y^2 + z^2}}
%			+ \frac{a_1}{x^2} + \frac{a_2}{y^2} + a_3\,\frac{4}{(x^2 + y^2 + z^2 + 1)^2}.
%	\]
%	Using modern computer algebra software, $D$, $\eta$ and $\nabla=\nabla^D$ are easily computed and it is verified that $\nabla=\nabla^D$ has vanishing torsion. Moreover, the Ricci tensor of $\nabla$ is obtained as
%	\[
%			\Ric^\nabla_{ij} = \frac{24}{(x^2+y^2+z^2+1)(x^2+y^2+z^2)}\,\sum_{1\leq i,j\leq 3}\,x^ix^j\,dx^i\,dx^j
%	\]
%	where we write $(x^1,x^2,x^3)=(x,y,z)$.
%	We also verify $\Ric=-2\eta$. The projective curvature tensor is obtained as
%	\[
%			\Proj^\nabla_{ij} = 0\,,
%	\]
%	as expected. We also verify
%	\[
%		\Delta V=-\frac12\Scal\,V
%	\]
%	where $\Delta$ is the Laplace operator and $\Scal$ the scalar curvature obtained from $\nabla$ using $g$ to raise indices.
%\end{example}
%
%\begin{example}
%	We consider two proper examples from~\cite{Escobar-Ruiz&Miller}, namely the system labeled `II' and `III'.
%	We quickly compute $D$ and confirm
%	\[ \eta=0\,,\quad\Riem^\nabla=0\,,\quad \Delta V=0\,. \]
%	where $\Delta$ is the Laplace operator obtained from $\nabla$ using $g$ to raise the index.
%\end{example}

%============================================================================%
\section{Examples}\label{sec:ex}

We illustrate the statements of the previous sections with a few examples from the literature.

%----------------------------------------------------------------------------%
\subsection{Non-degenerate potentials}

We reconsider some well-studied examples on constant curvature spaces, which are investigated in~\cite{KKM2006,KKM06a,Capel_phdthesis}.

%\todo[inline]{Compute the curvature for each of them, confirm $\Proj=0$ and list $\Scal$}

\begin{example}
	We compute $B_{ij}$ for the Smorodinski-Winternitz system in dimension~$n$, $n\geq3$, which has the potential
	\[
		V(x^1,\dots,x^n)=a_0\,\sum_{k=1}^n(x^k)^2 + \sum_{k=1}^n\frac{a_k}{(x^k)^2}\,.
	\]
	We find $\tau_{ij}=0$, and
	\[
		B=\frac{2e_{ij}}{x^ix^j}\,dx^i\,dx^j
	\]
	where $e_{ii}=+1$ and $e_{ij}=-1$ if $i\ne j$.
	We then straightforwardly verify the formulas
	\[
		\Ric_{ij}=-\frac1{n-1}(B_{ij}-g_{ij}\tr(B))\,.
	\]
	and
	\[
		\Riem_{aijk} =  \frac1{n-1}g_{ij}B_{ak}-\frac1{n-1}g_{ik}B_{aj}\,.
	\]
	for the Ricci and Riemann tensors of $\nabla$.
\end{example}

%----------------------------------------------------------------------------%
\subsection{Semi-degenerate potentials}

We investigate semi-degenerate potentials found in~\cite{Escobar-Ruiz&Miller}. These are of dimension~3, but some of them can be extended to arbitrary dimension.

%For instance, for the (proper) system with label `II', one easily obtains $D$ and verifies that $\Riem=0$ and $\eta=0$.
%%	We also write down $d$ and its function.
%
%Similarly, for the system with label `I', $D$ is straightforwardly obtained. The objects $\eta$, $\Riem$ and $\Ric$ are then obtained, and the formula~\eqref{eq:Eta} is verified:

\begin{example}
	We consider a conformal example from~\cite{Escobar-Ruiz&Miller}, namely the system labeled `I'.
	The potential for the system `I' is
	\[
	V= a_0\,\frac{x^2 + y^2 + z^2 - 1}{(x^2 + y^2 + z^2 + 1)^2\sqrt{x^2 + y^2 + z^2}}
	+ \frac{a_1}{x^2} + \frac{a_2}{y^2} + a_3\,\frac{4}{(x^2 + y^2 + z^2 + 1)^2}.
	\]
	One easily computes $D$, $\eta$ and $\nabla=\nabla^D$ and verifies that $\nabla=\nabla^D$ has vanishing torsion. Moreover, the Ricci tensor of $\nabla$ is obtained as
	\[
	\Ric_{ij} = \frac{24}{(x^2+y^2+z^2+1)(x^2+y^2+z^2)}\,\sum_{1\leq i,j\leq 3}\,x^ix^j\,dx^i\,dx^j
	\]
	where we write $(x^1,x^2,x^3)=(x,y,z)$.
	We also verify $\Ric=-2\eta$. The projective curvature tensor is confirmed to vanish, $\Proj^\nabla_{ij} = 0$,
	as expected. We also verify
	\[
	\Delta V=-\frac{\Scal}{2}\,V
	\]
	where $\Delta$ is the Laplace operator and $\Scal$ the scalar curvature for $\nabla$ (using $g$ to raise indices).
\end{example}

\begin{example}
	We consider two proper examples from~\cite{Escobar-Ruiz&Miller}, namely the system labeled `II' and `III'.
	We quickly compute $D$ and confirm
	\[ \eta=0\,,\quad R=0\,,\quad \Delta V=0\,. \]
	where $R=R^\nabla$ and where $\Delta=\Delta^{\nabla,g}$.
\end{example}

%============================================================================%
\section{Discussion and conclusion}\label{sec:discussion}

This paper has argued that the Ricci-symmetric, torsion-free connection naturally defined by the primary structure tensor of a (proper) non- or semi-degenerate potential gives rise to a flat projective structure.
This suggests that, in addition to conformal geometry \cite{KSV2024}, projective differential geometry is another important geometric structure underpinning superintegrable systems. Moreover, this corroborates that the integrability conditions for the ``Wilczynski equations'' \eqref{eq:Wilczynski} and~\eqref{eq:semi-Wilczynski} are inherently \emph{geometric} conditions.

The only remaining obstacle for (maximal) superintegrability then is the existence of (sufficiently many suitable) Killing tensor fields $K_{ij}$ satisfying~\eqref{eq:ddW}, where ``suitable'' alludes to the fact that $2n-1$ \emph{functionally} independent integrals of motion have to exist for (maximal) superintegrability.
Resubstituting the Wilczynski equation, i.e.~\eqref{eq:Wilczynski} respectively~\eqref{eq:semi-Wilczynski}, into the Bertrand-Darboux condition~\eqref{eq:ddW}, we obtain, cf.~also \cite{KSV2023,Jeremy},\medskip

\begin{minipage}{0.45\textwidth}\centering
	\textbf{In the non-degenerate case:}
	\[
		\young(j,k)\left[ \nabla_kK\indices{^a_j}V_a + K\indices{^b_j}T\indices{_{kb}^a}V_a \right] = 0
	\]
\end{minipage}
\begin{minipage}{0.45\textwidth}\centering
	\textbf{In the semi-degenerate case:}
	\[
		\young(j,k)\left[ \nabla_kK\indices{^a_j}V_a + K\indices{^b_j}D\indices{_{kb}^a}V_a \right] = 0
	\]
\end{minipage}
\medskip

\noindent Hence, in either case, we have an equation of the form
\[
	\nabla_kK_{ij} = P\indices{_{ijk}^{ab}}K_{ab}
\]
where $P\indices{_{ijk}^{ab}}$ is symmetric in $(i,j,k)$ and $(a,b)$.
Note that $P$ depends on $T$ in the non-degenerate case, and on $D$ in the semi-degenerate one, but we shall use the same symbol here for conciseness.
The Ricci identity, in both cases, thus has the form, where $R$ is the curvature tensor of $\nabla$,
\begin{equation}\label{eq:P.integrability}
	\underbrace{ \young(k,l)\left[
		\nabla_lP\indices{_{ijk}^{ab}}+P\indices{_{ijk}^{pq}}P\indices{_{pql}^{ab}}
		- \frac14\young(ij) (R\indices{^a_{ikl}}g\indices{^b_j}-R\indices{^b_{ikl}}g\indices{^a_j})
		\right] }_{\text{depends on $\nabla$ and $g$ only}} K_{ab} = 0\,.
\end{equation}
For a given Riemannian metric $(M,g)$ a torsion-free, Ricci-symmetric connection $\nabla$ will be called \emph{admissible} (for non-degeneracy) if it is as in (i) of Theorem~\ref{thm:proper.admissible.nondeg} resp.\ Theorem~\ref{thm:proper.admissible.semideg}. It is called \emph{admissible} (for semi-degeneracy) if it is as in (ii) of the theorems.

%----------------------------------------------------------------------------%
\subsection*{Acknowledgements}
The author would like to thank Jonathan Kress, Vladimir Matveev, Vicente Cortes, Konrad Schöbel and Jeremy Nugent for discussions.
This research has been funded by the German Research Foundation (Deutsche Forschungsgemeinschaft), project number 540196982.

%============================================================================%
%\clearpage
\bibliographystyle{amsalpha}
\bibliography{refs}

\providecommand{\bysame}{\leavevmode\hbox to3em{\hrulefill}\thinspace}
\providecommand{\MR}{\relax\ifhmode\unskip\space\fi MR }
\providecommand{\MRhref}[2]{%
  \href{http://www.ams.org/mathscinet-getitem?mr=#1}{#2}
}
\providecommand{\href}[2]{#2}
\begin{thebibliography}{KKPM01}

\bibitem[Ber57]{bertrand_1857}
J.~M. Bertrand, \emph{M\'emoire sur quelques-unes des forms les plus simples
  qui puissent pr\'esenter les int\'egrales des \'equations diff\'erentielles
  du mouvement d'un point mat\'eriel}, J. Math. Pure Appl. \textbf{II} (1857),
  no.~2, 113--140.

\bibitem[Cap14a]{Capel}
Joshua~J. Capel, \emph{Classification of second-order
  conformally-superintegrable systems}, Ph.D. thesis, School of Mathematics and
  Statistics, University of New South Wales, 2014.

\bibitem[Cap14b]{Capel_phdthesis}
Joshua~Jordan Capel, \emph{Classification of second-order
  conformally-superintegrable systems}, Ph.D. thesis, School of Mathematics and
  Statistics, University of New South Wales, 6 2014.

\bibitem[Dar01]{darboux_1901}
G.~Darboux, \emph{Sur un probl\`eme de m\'echanique}, Archives N\'eerlandaises
  (ii), vol.~6, 1901, pp.~371--376.

\bibitem[ERJ17]{Escobar-Ruiz&Miller}
M~A Escobar-Ruiz and Willard~Miller Jr, \emph{Toward a classification of
  semidegenerate 3d superintegrable systems}, Journal of Physics A:
  Mathematical and Theoretical \textbf{50} (2017), no.~9, 095203.

\bibitem[Jer]{Jeremy}
Jeremy Nugent: PhD thesis (UNSW Sydney), private communication.

\bibitem[KKJ06]{KKM2006}
E.G. Kalnins, J.M. Kress, and W.~Miller Jr., \emph{Classification of
  superintegrable systems in three dimensions}, Quantum Theory and Symmetries
  IV (2006).

\bibitem[KKM05a]{KKM05a}
E.~G. Kalnins, J.~M. Kress, and W.~Miller, Jr., \emph{Second-order
  superintegrable systems in conformally flat spaces. {I}. {T}wo-dimensional
  classical structure theory}, J. Math. Phys. \textbf{46} (2005), no.~5,
  053509, 28.

\bibitem[KKM05b]{KKM05b}
\bysame, \emph{Second order superintegrable systems in conformally flat spaces.
  {II}. {T}he classical two-dimensional {S}t\"ackel transform}, J. Math. Phys.
  \textbf{46} (2005), no.~5, 053510, 15.

\bibitem[KKM05c]{KKM05c}
\bysame, \emph{Second order superintegrable systems in conformally flat spaces.
  {III}. {T}hree-dimensional classical structure theory}, J. Math. Phys.
  \textbf{46} (2005), no.~10, 103507, 28.

\bibitem[KKM06]{KKM06a}
\bysame, \emph{Second order superintegrable systems in conformally flat spaces.
  {IV}. {T}he classical 3{D} {S}t\"ackel transform and 3{D} classification
  theory}, J. Math. Phys. \textbf{47} (2006), no.~4, 043514, 26.

\bibitem[KKM07]{Kalnins&Kress&Miller}
E.~G. Kalnins, J.~M. Kress, and W.~{Miller, Jr.}, \emph{Nondegenerate 2{D}
  complex {E}uclidean superintegrable systems and algebraic varieties}, Journal
  of Physics A: Mathematical and Theoretical \textbf{40} (2007), 3399--3411.

\bibitem[KKM18]{KKM18}
Ernest~G Kalnins, Jonathan~M Kress, and Willard Miller, \emph{Separation of
  variables and superintegrability}, 2053-2563, IOP Publishing, 2018.

\bibitem[KKPM01]{Kalnins&Kress&Pogosyan&Miller}
E.~G. Kalnins, J.~M. Kress, G.~S. Pogosyan, and W.~{Miller, Jr.},
  \emph{Completeness of multiseparable superintegrability in two-dimensional
  constant curvature spaces}, Journal of Physics A: Mathematical and General
  \textbf{34} (2001), 4705--4720.

\bibitem[Kre07]{Kress07}
J.~M. Kress, \emph{Equivalence of superintegrable systems in two dimensions},
  Physics of Atomic Nuclei \textbf{70} (2007), no.~3, 560--566.

\bibitem[KS19]{Kress&Schoebel}
Jonathan Kress and Konrad Schöbel, \emph{An algebraic geometric classification
  of superintegrable systems in the euclidean plane}, Journal of Pure and
  Applied Algebra \textbf{223} (2019), no.~4, 1728--1752.

\bibitem[KSV23]{KSV2023}
Jonathan Kress, Konrad Schöbel, and Andreas Vollmer, \emph{An algebraic
  geometric foundation for a classification of superintegrable systems in
  arbitrary dimension}, Journal of Geometric Analysis \textbf{33} (2023),
  no.~360.

\bibitem[KSVMP]{KSV2024}
\bysame, \emph{An algebraic geometric foundation for a classification of
  superintegrable systems in arbitrary dimension}, To appear in CIMP.

\bibitem[Shi07]{Shima}
Hirohiko Shima, \emph{The geometry of hessian structures}, World Scientific,
  2007.

\end{thebibliography}
\medskip

\end{document}